\documentclass[graybox]{svmult}
\usepackage{mathptmx}       % selects Times Roman as basic font
\usepackage{helvet}         % selects Helvetica as sans-serif font
\usepackage{courier}        % selects Courier as typewriter font
\usepackage{type1cm}        % activate if the above 3 fonts are
\usepackage{makeidx}         % allows index generation
\usepackage{graphicx}        % standard LaTeX graphics tool
\usepackage{multicol}        % used for the two-column index
\usepackage[bottom]{footmisc}% places footnotes at page bottom
\usepackage{url}

\usepackage{amsfonts}
\usepackage{amsmath}
\usepackage{amssymb}
\newcommand{\Asp}{{\boldsymbol A}}     %A-space
\newcommand{\CORdN}{{\big( \COsp(\Rst^d), \, \|\ebbes\|_\infty \big)}}     %CO(Rd) Normed
\newcommand{\COsp}{{\Csp_{\negthinspace 0}}}     %C_0-space
\newcommand{\Rst}{{\mathbb R}}     %R set (bold)
\newcommand{\ebbes}{\mbox{$\,\cdot\,$}}     %something (etwas)
\newcommand{\Csp}{{\boldsymbol C}}     %C-space (bold C)
\newcommand{\CbRd}{{\Cbsp(\Rdst)}}     %Cb(Rd) bounded continuous on Rd
\newcommand{\Cbsp}{{\Csp_{\negthinspace b}}}     %Cb-space
\newcommand{\Rdst}{{{\Rst^d}}}     %R^d
\newcommand{\Cst}{{\mathbb C}}     %complex numbers
\newcommand\Hilb{\mathcal H}     %mathcal H (Hilbert space)
\newcommand{\LiG}{{\Lisp(G)}}     %L1(G)
\newcommand{\Lisp}{{\Lsp^1}}     %L1-space
\newcommand{\LiGN}{\big( \LiG, \, \|\ebbes\|_1 \big)}     %L1(G) Normed
\newcommand{\LiRd}{{\Lisp \nth (\Rst^d)}}     %L1(Rd)
\newcommand\nth{\negthinspace}     %negthinspace
\newcommand{\LiRdN}{\big( \LiRd, \, \|\ebbes\|_1 \big)}     %L1(Rd) Normed
\newcommand\Liloc{ {\Lsp^1_{ \negthinspace \it{loc}} } }     %Liloc (new version)
\newcommand{\Lsp}{{\boldsymbol L}}     %L-space (L^1, etc.)
\newcommand{\LtR}{{\Ltsp(\Rst)}}     %L2(R)
\newcommand{\Ltsp}{{\Lsp^2}}     %L2-space
\newcommand{\LtRd}{{\Ltsp(\Rst^d)}}     %L2(Rd)
\newcommand{\LtRdN}{\big( \LtRd, \, \|\ebbes\|_2 \big)}     %L2(Rd) - normed
\newcommand{\MbRd}{{\Mbsp(\Rst^d)}}     %bounded measures on R^d
\newcommand{\Mbsp}{{\Msp_{\negthinspace b}}}     %Mb-space
\newcommand\MbRdN{{(\Mbsp(\Rst^d), \| \ebbes \|_\Mbsp )}}     %bounded measures with norm
\newcommand{\Msp}{{\boldsymbol M}}     %M-bold
\newcommand{\MiRdN}{{ (\Msp^1(\Rdst), \| \ebbes \|_{\Msp^1}) }}     %M^1(Rd) Normed
\newcommand\MsTt{{M_s T_t}}     %M_s T_t
\newcommand{\Nst}{{\mathbb N}}     %N natural numbers
\newcommand{\Rdhat}{{\widehat {\Rdst}}}     %Rd hat
\newcommand{\Rdsth}{{\widehat{\Rst}^d}}     %Rd-set hat, frequency domain
\newcommand{\SOG}{{\SOsp(G)}}     %SO(G)
\newcommand{\SOsp}{{\Ssp_{\negthinspace 0}}}     %SOsp
\newcommand{\SOGN}{\big( \SOG, \|\ebbes\|_\SOsp \big)}     %SO(G) Normed
\newcommand{\SOGTrRd}{{ (\SOsp,\Ltsp,\SOPsp)(\Rdst) }}     %SO-Gelfand Triple on Rd
\newcommand{\SOPsp}{{\Ssp_{\negthinspace 0}'}}     %S_0' space (S0-prime space)
\newcommand{\SOPN}{(\SOPsp , \| \ebbes \|_{\SOPsp} ) }     %S0' normed
\newcommand{\SOPR}{{\SOPsp(\Rst)}}     %SO'(R) S0-prime
\newcommand{\SOPRN}{(\SOPR , \| \ebbes \|_{\SOPsp} ) }     %S0'(R) normed
\newcommand{\SOPRd}{{\SOPsp(\Rst^d)}}     %S0' (Rd)
\newcommand{\SOPRdN}{(\SOPRd , \| \ebbes \|_{\SOPsp} ) }     %S0'(Rd) normed
\newcommand{\SOPnorm}[1]{{\lVert #1 \rVert_\SOPsp}}     %SOP-norm
\newcommand{\Ssp}{{\boldsymbol S}}     %S-space
\newcommand{\SORd}{{\SOsp(\Rst^d)}}     %S0(Rd)
\newcommand{\SORdN}{\big( \SORd, \|\ebbes\|_\SOsp \big)}     %S0(Rd) normed
\newcommand{\SOnorm}[1]{{\lVert #1 \rVert_\SOsp}}     %SO-norm (arg)
\newcommand{\Scsp}{{\boldsymbol{\mathcal S}}}     %Schwartz space (mathcal)
\newcommand{\ScRd}{{\Scsp(\Rst^d)}}     %Schwartz space Rd (recommended version)
\newcommand{\TFd}{{{ \Rdst \times \Rdsth }}}     %TF over Rd
\newcommand{\Tst}{\mathbb{T}}     %torus (1D torus): unitary group
\newcommand{\Ust}{\mathbb{U}}     %torus (1D torus): unitary group
\newcommand{\Zdst}{{\Zst^d}}     %Zd (official version)
\newcommand{\Zst}{{\mathbb Z}}     %Z mathbb
\newcommand\abZZd{ {a \Zdst \times b \Zdst}}     %aZd x b Zd (alternative version)
\newcommand{\cG}{\mathscr{G}}     %Calligraphic Letter
\newcommand\cGd{{\widehat{\cG}}}     %cG dual group
\newcommand\clam{{c_\lambda}}     %c_lambda
\newcommand\dlam{{d\lambda}}     %d lambda
\newcommand{\gd}{{\widetilde{g}}} %%mdef     %gd:wtilde (g), dual Gabor atom
\newcommand\hatf{{\widehat{f}}}     %f hat (cf. tilde f)
\newcommand\hatsi{{\widehat{\sigma}}}     %hat of sigma
\newcommand\hkr{\hookrightarrow}     %hookrightarrow (short)
\newcommand{\kiZd}{{{k \in \Zdst}}}
\newcommand{\lainLa}{{\lambda{\in}\Lambda}}     %lambda in Lambda
\newcommand\liLam{{\lisp(\Lambda)}}     %l1(Lambda)
\newcommand{\lisp}{{\lsp^1}}     %l1-space
\newcommand{\lsp}{{\boldsymbol\ell}}     %ell-space (e.g. \ell^p)
\newcommand\pilamg{ \pi(\lambda) g}     %pi(lambda) g
\newcommand{\projtens}{{ \widehat{\otimes} }}     %projective tensor
\newcommand\siSOP{{ \sigma \in \SOPsp }}     %sigma in SO prime
\newcommand\signN{{ (\sigma_n)_{n \in \Nst}} }     %sigma_n, n in N sequence
\newcommand\signo{ \sigma_0 = \wst-\mbox{lim}_n \sigma_n }     %sigma_0 = wst-lim s_n
\newcommand{\sumkZd}{\sum_{\kiZd}}     %sum k in Zd
\newcommand{\sumlaLa}{\sum_{\lambda\in\Lambda}}     %sum_{lambda in Lambda}
\newcommand{\sumnZ}{\sum_{n\in\Zst}}     %sum n in Z
\newcommand\sumnZd{\sum_{n\in\Zdst}}     %sum n in Zd
\newcommand{\supnorm}[1]{{\lVert #1 \rVert_\infty}}     %supnorm arg1 #1
\newcommand{\supp}{\operatorname{supp}}     %support (of a function or distribution)
\newcommand{\tensor}{\otimes}     %tensor (o-times)
\newcommand{\wdash}{{ w^* \negthinspace \mbox{-}}}     %w*-dash
\newcommand\wst{ w^{*} }     %weak star
\newcommand\wstd{{w^* \negthinspace -}}     %wst dash
\newcommand{\wwst}{{\wdash \wdash}}     %wst-to-wst continuous

\def\FeiAdr{Hans G. Feichtinger, Faculty of Mathematics, University of Vienna \newline and Acoustic Research Institute, OEAW, AUSTRIA}
\def\cG{G}

\def\normta#1#2{{  \| {#1}   \|_{#2} \, }} 

\begin{document}

\title{Fourier Analysis via Mild Distributions: \newline
  Group Theoretical Aspects}
\author{Hans G. Feichtinger} %order: first (given) name family name (no abbreviations!)
\institute{Hans G. Feichtinger \at Faculty of Mathematics, University
\at \email{hans.feichtinger@univie.ac.at} % %order: first (given) name family name (first name abbreviated!)
% \and Name of Second Author \at Name, Address of Institute
}  %Provide only (preferably university) e-mail of corresponding author!

\maketitle

\abstract{ This note aims at providing a guide towards the use of mild distributions, or more generally the concept of Banach Gelfand Triples in the context of Fourier Analysis, both in the classical and the application oriented sense. }
%Do not cite articles in abstract by using \cite{} command. Use e.g. Maes and Van Bockstal (Fract Calc Appl Anal 26(4) 2023, Lemma 5)


\section{Introduction}  % Fei:  [fe24-5]

The Segal algebra $\SORdN$, also known as  {\it Feichtinger's algebra},
has been already introduced in 1979 by the author. It was published in \cite{fe81-2}, and \cite{ja18}  provides a recent survey. It turned out to be
an important tool in Gabor Analysis and in Time-Frequency Analysis
in general. It also appears in  \cite{gr01} for a mathematical introduction to this field. Among others windows in $\SORd$ guarantee the convergence
of Gabor series for general $f \in \LtRd$ and thus is also a good basis
for the study of Gabor multipliers.

When it comes to the description of operators in general the triple,
also called ``THE Banach Gelfand Triple'' $\SOGTrRd$  (BGT) turns out to be
quite useful. It consists of the space $\SORdN$ (which can be viewed as a
Banach algebra of test functions with pointwise multiplication), the
Hilbert space $\LtRdN$ and the dual space $\SOPRdN$, the spBGace of
``{\it mild distributions}''. The form a chain of continuous inclusions:
$$ \SORdN \hkr \LtRdN \hkr \SOPRdN. $$
The (isometric) Fourier invariance of $\SORdN$ and consequently of its
dual allows to describe the Fourier transform in the BGT-setting as
a {\it unitary BGT-automorphism}. At the level of $\Hilb = \LtRd$ this
corresponds just to Plancherel's Theorem, but in the context of
$\SOPRd$  it is characterized by mapping pure
frequencies $\chi_s(t) = \exp(2 \pi i s t)$ into the corresponding
Dirac measures $\delta_x \in \SOsp(\widehat{\Rdst})$.
The concept of BGTs goes back to  
  \cite{feko98} and \cite{fezi98} (which appeared in \cite{fest98}),
and a first summary was presented in \cite{cofelu08}, see also the master
thesis \cite{ba10-2}. It appears as modulation space $\MiRdN$ in 
\cite{gr01} and many other places (see \cite{fe03} and \cite{fe06} for
the connection to modulation spaces, or the books \cite{coro20} and \cite{beok20}). The ``inner kernel theorem'' is discussed in \cite{feja22}. 

On the other hand engineers learn about {\it discrete} or {\it continuous} signals (time series, indexed by $\Zst$, or continuous functions $f(t)$),
where $t \in \Rst$ is interpreted as time. A signal can be {\it periodic}
 or {\it non-periodic} (which typically means that it is
somehow decaying at infinity, such as signals of finite energy, modelled as
$f \in \LtR$). Correspondingly, there is a particular form of Fourier transform
for each case (including the multi-dimensional ones). For discrete and
periodic signals (which can be identified with the Euclidean space $\Cst^N$,
for some $N \in \Nst$) one has to use the DFT, implemented as FFT (Fast Fourier Transform). Sampling or restriction to a subgroup corresponds to periodization
(or integration) along the orthogonal subgroup, and so on. The aim of this
short note is to point out how these transitions can be described using
the theory of mild distributions.

Let us just recall that $\SORd$ is an (isometrically) Fourier invariant
Banach space, a Banach algebra with respect to  pointwise multiplication {\it and} convolution. Functions in $\SORd$ are (absolutely Riemann) integrable
and bounded, continuous functions on $\Rdst$. Thus $\SORdN$
is a Banach algebra continuously embedded into $\CORdN$ (with pointwise
multiplication) and into $\LiRdN$ (a Banach algebra with respect to convolution). Both the classical Fourier inversion theorem holds true
in the context of $\SORd$ (in a pointwise sense). The Schwartz space
$\ScRd$ is continuously embedded into $\SORdN$ as a dense subspace. Moreover,
Poisson's formula holds true for arbitrary lattices (discrete and cocompact
subgroups of $\Rdst$). In the classical form it reads $\sumkZd f(k) = \sumnZd \hatf(n)$, with absolute convergence
 on both sides (Cor.12.1.5 in \cite{gr01}). 

\section{Mild Distributions}

% For those familiar with the theory of
Starting from L.~Schwartz's theory of 
 {\it tempered distributions}
we can define {\it mild distributions} by a direct boundedness
condition (equivalent to the boundedness of the Short-Time Fourier
transform, the STFT, see \cite{gr01}). The finiteness
of the following norm defines an {\it admissible norm} for $\SOPRd$,  in the sense of H.~Triebel (\cite{tr16}): 
\begin{equation}\label{SOPnormAdm}
   \SOPnorm{\sigma}  := 
   \sup_{(t,s) \in \TFd}  \sigma(\MsTt g_0) < \infty,
\end{equation}
with the usual conventions $\lambda = (t,s), \pi(\lambda) = \MsTt$, with
$T_t f(x) = f(x-t)$ and $M_sf(x) = \exp(2 \pi i s x) f(x)$, $x,t,s \in \Rdst$,
 or $(t,s) \in \TFd$,  see \cite{gr01} for details. The  
 STFT of $f$ with window $g$ on $\TFd$ is given by 
$V_gf(\lambda) = \langle f , \pilamg \rangle$.  
% Definition STFT on $\ScPRd$.

% By observing that the STFT (with any Schwartz window $g \in \ScRd$)
% is a continuous function which has at most polynomial growth  it is
% clear that those tempered distributions with a bounded STFT form
% a subspace. % So let us define {\it mild distributions} next:

%%% Definition: Mild distributions as a (Banach) subspace of
%%% the space of tempered distributions, with the sup-norm.

Alternatively one can introduce mild distributions by a suitable
sequential approach, which is of course the analogue of the Lighthill
approach to tempered distributions. While the equivalence of the
sequential approach to the topological approach to generalized functions
is highly non-trivial in the Schwartz case (based on results by G.~Temple
established around 1955) this is relatively easy in the context of
mild distributions (see \cite{fe20-1}).

% \cite{fe20-1}

\subsection{Convergence of Mild Distributions}

It turns out that norm-convergence in $\SOPRN$ is too strict, like 
norm convergence in $\MbRdN$, the space of bounded measures. One has
$ \normta{ \delta_x - \delta_y} {\SOPsp} = 2 $ for $x \neq y$. Hence it
is important to replace it by the so-called $\wstd$convergence, which can
be described by sequences:
% \begin{definition} \label{SOPconv002}  xxx  \end{definition}
%
\begin{lemma}  \label{wstSOPsequ1}
Convergence in the $\wstd$topology   can be described by sequential 
convergence.  
A sequence $\signN$ is $\wstd$convergent  with limit
 $\sigma_0 \in \SOsp$ if and only if  
\begin{equation}\label{signconv02}
   \lim_{n \to \infty} \sigma_n(f) = \sigma_0(f), \quad f \in \SOsp.
\end{equation}
We write $\signo$ in this case.
\end{lemma}
\begin{proof}
That the convergence can be described by sequences and not just nets
can be justified by the fact that $\SORdN$ is a separable Banach space.
On the other hand, pointwise convergence of a sequence of linear functionals
implies its boundedness by the Banach-Steinhaus Theorem.
\end{proof}
\begin{remark}
Due to the fact that according to our understanding this type of convergence
is the most relevant in the space $\SOPsp$ of mild distribution we will aslo
speak of {\it mild convergence}, as oppposed to {\it strong} (i.e., norm) convergence in $\SOPN$.
\end{remark}
There are various useful equivalent description of mild convergence
\begin{theorem} \label{mildconvchar03}
The following conditions provide equivalent descriptions of mild
convergence for bounded sequences $\signN$  in $\SOPRd$:
% equivalent conditions is satisfied:
% A bounded sequence $\signN$ is mildly convergent if and only one of the following
% equivalent conditions is satisfied:
\begin{enumerate}
  \item For some/any non-zero $g \in \SORd$ one has
      pointwise convergence of $V_g\sigma_n(\lambda)$, $\lambda \in \TFd$;
  \item For some/any non-zero $g \in \SORd$ one has
     uniform convergence over compact subsets of $\TFd$;
  \item For some lattice $\abZZd$ with $ab < 1$ one has
      pointwise convergence of the canonical (minimal norm)
      Gabor coefficients $(c^{(n)}_{k,l})$ providing the
      Gabor expansion
      $  \sigma_n = \sum_{k,l} c^{(n)}_{k,l} M_{bl} T_{ak} g_0,$
      with $g_0(t) = \exp(- \pi |t|^2)$, the standard Gaussian.
\end{enumerate}
\end{theorem}
\begin{proof} We only provide the key arguments here.

Since the sequence $\signN$ is bounded it is enough to relate the pointwise convergence of $V_g\sigma(\lambda) = \sigma(\overline{\pilamg})$ to the
convergence of finite linear combinations of TF-shifts of $g$. Due to well-known atomic characterization of $\SORdN$ (for any nonzero $g \in \SORd$)
one concludes that one has $\wstd$convergence.

In a similar way a standard approximation argument allows to derive uniform
convergence of the STFT from the pointwise convergence, using the fact that
$\lambda \to \pilamg$ is a continuous mapping from $\TFd$ to $\SORdN$, and thus
for any compact set $K \subset \TFd$ the range of this mapping is a compact
subset of $\SORdN$.

The third point relies on the fact that the minimal norm coefficients are
obtained by choosing $c^{(n)}_{k,l} = V_{\gd}\sigma_n(ak,bl)$, where
$\gd$ is the canonical dual Gabor atom for $g$ with respect to the TF-lattice
$\abZZd$. Deep results in Gabor analysis imply that $g_0$ generated a Gabor
frame for this lattice if $ab < 1$, with a dual window also in $\ScRd$. Alternatively, another key result of Gabor analysis implies that
$gd \in \SORd$  if some $ g \in \SORd$ generates a Gabor frame. Anyway,
the canonical Gabor coefficients are just samples of the $V_{\gd}\sigma$.
It is well known that they are bounded with 
$$ \supnorm{V_g\sigma} \leq C \SOPnorm{\sigma}, \quad \sigma \in \SOPsp. $$

Finally, coordinatewise convergence of the canonical
(or even non-canonical) Gabor coefficients (with respect to any dual window
in $\SORd$) implies locally uniform convergence of the corresponding
STFTs $V_g \sigma_n$ is an easy exercise.
\end{proof}

\subsection{Poisson's Formula and Dirac Combs}

The paper \cite{fe24} gives a detailed study of Dirac combs and their role for sampling (along a discrete lattice in $\Rdst$), and that this corresponds to
periodization of the Fourier transform of any function $f \in \SORd$.
The so-called BUPUs (Bounded Uniform Partitions of Unity) discussed in
\cite{fe24-2} allow to return approximately from the sampled version of
a function $f \in \SORd$ to a continuous approximation of $f$, if the
sampling lattice is fine enough. Here sampling corresponds to the restriction
$R_H(f)$ of $f$ to a subgroup. The adjoint mapping is given by 
$R^*_H \mu(f) = \mu(R_H f)$, for $\mu \in \MbRd$. Clearly 
$\supp(R^*_H(\mu)) \subseteq H$. Via the Fourier transform restriction
mapping and the canonical mapping $T_{H^\perp}$ correspond to each other. 
Also discrete and periodic signals belong to $\SOPRd$ and their (distributional) FT can be realized using the FFT algorithm.

\section{Group Theoretical Considerations}

First we  recall the characterization of elements
 $\siSOP$ with special properties. For their description we need two
 standard definitions: Given a subgroup $H \lhd \Rdst$  we call a function
 $f \in \CbRd$ $H$-periodic if one has $T_h f = f, h \in H$. The sequential
 approach to $\SOPsp$ allows to extend this notion to all of $\SOPRd$.
 Equivalently we can define translation of distributions by $T_{h}\sigma(f) := \sigma(T_{-h}f), h \in H, f \in \SORd$. 

 The support $\supp(\sigma)$ of $\siSOP$ is defined as the complement of
 the open set of ``trivial action''. A point $z \in \Rdst$ is a point of (locally) trivial action of $\sigma$ if for a sufficiently small ball $B_\delta(z)$ around  $z$ one has $\sigma(k)=0$ for all the functions $k \in \SORd$ with  $\supp(k) \subseteq B_\delta(0)$.

\vspace{2mm} 
 With this terminology we are ready to provide the following characterizations.

 \begin{proposition}  \label{charsuppH}
 Given a lattice $ \Lambda \lhd \Rdst$ we have: \newline 
 A distributions
 $\sigma \in \SOPRd$ satisfies $\supp(\sigma) \subset \Lambda$ if and only
 if it is a weighted Dirac comb, for some bounded sequence ${\bf c} = (\clam)_{\lainLa}$, i.e.,  $\sigma = \sumlaLa \clam \dlam$.
 \end{proposition}

% UUUU
The basis for Abstract Harmonic Analysis (AHA) over 
 locally compact Abelian (LCA) groups $\cG$  is the existence of sufficiently 
 many 
% also has a {\it dual group} of
 {\it characters}, i.e., continuous homomorphism
from $\cG$ into the torus group $\Ust$ (unitary elements in $\Cst$, i.e.
the unit circle with complex multiplication). These {\it characters}
form the {\it dual group} $\cGd$, with respect to pointwise multiplication.
$\cGd$ can be endowed with a natural, locally compact  topology, thus
turning into another LCA group.

 It is not difficult to verify that there is a natural embedding of the group
$\cG$ into the dual group of $\cGd$, via $x(\chi) = \chi(x)$.
By Pontryagin's Theorem this embedding is always surjective, hence
a bijection. In this sense every LCA group is a dual group of another
LCA group (namely $\cGd$). The Fourier Transform (if we take the
integral version) can then be defined with the help of the Haar measure
as $$ \hatf(\chi) = \int_{\cG} f(x) \overline{\chi(x)} dx, \quad f \in \LiG. $$

 For the case of $\cG = \Rdst$ the dual group consists of all the characters
 $\chi_s(x) = \exp(2 \pi i s t)$, with $s,t \in \Rdst$ and the product being
 interpreted as the Euclidean scalar product $st = \sum_{k=1}^d s_k t_k$. The abstract topology on $\cGd$ is equivalent to the  
 Euclidean topology on $\Rdst$, which allows to identify $\Rdhat$ with $\Rdst$
 ($\chi_s$ corresponds to $s \in \Rdst$).

 Looking into the category of LCA groups (with objects $\cG$) with continuous homomorphisms between such groups as {\it morphism}
% [MacLane, Saunders] Categories for the working mathematician. 4th corrected printing.
\cite{ma88} it is natural to study the connections between different groups
and identify the subgroups and quotient groups of such a LCA group, and
verify their consequences for the Fourier transforms available in the
different settings.

In the context of AHA, making use of integration theory such considerations
have been undertaken in \cite{re68} when he studies Weil's Formula
(cf. \cite{ru62}) which implies that for any subgroup $H \lhd \cG$ 
the {\it canonical mapping} $T_H$, given by
\begin{equation} \label{THmap001}
{T_H f(\dot{x})=\int_H  f(x\xi) d\xi},
\quad \dot{x} = x+H,  f \in \LiG,
\end{equation}
maps $\LiGN$ boundedly (with norm $1$) onto $\Lisp(G/H)$. Moreover
the dual group of $G/H$ can be identified with  $H^\perp \lhd \widehat{G}$,
the  orthogonal subgroup to $H$.  % $\widehat{H}$

%  We will work in the Euclidean setting in this short note.

\section{Subgroups and Quotients}

One of the important properties of the functor, which assigns to each
LCA group the Segal algebra $\SOGN$ collected in the following theorem.
Again we restrict our attention to subgroups/quotients in $\Rdst$.

First let us observe that we have $\SOsp(\Zdst) = \lisp(\Zdst)$ and
(via the Fourier transform) $\SOsp(\Tst) =\Asp(\Tst)$, Wiener's algebra of absolutely convergent Fourier series.
% \newpage
\begin{theorem} \label{functRd01}
\begin{enumerate}
  \item For a subgroup $H \lhd \Rdst$ $R_H$ maps $\SORdN$  
  onto $\SOsp(H)$;
  \item For $H = \Rst^k \times  \{0\} \lhd \Rdst$ (with $k < d$) an
  extension operator can be obtained by putting $E(f) = f \tensor g$,
  for $g \in \SOsp(\Rst^{d-k})$ with $g(0)=1$, for any $f \in \SOsp(\Rst^k)$,
  where $ f\tensor g(x,y) = f(x)g(y)$.
   \item In fact, $\SORd = \SOsp(\Rst^k) \projtens \SOsp(\Rst^{d-k})$,
   with the projective tensor product norm.
  \item Consequently $\sumkZd |f(k)| \leq C \SOnorm{f}, f \in \SORd$,
          for $C > 0$.
  \item The extension from $\liLam$ (where $\Lambda \lhd \Rdst$ is any
  lattice inside of $\Rdst$) can be obtained by semidiscrete convolution,
  using suitable bump-functions $\varphi \in \SORd$ with small support and
  $\varphi(0) =1$:
  $$    E({\bf c}) = \sumkZd c_k T_k \varphi
   = \left(\sumkZd  c_k \delta_k \right ) \ast \varphi, \quad \mbox{for}
   \,\,  {\bf c} =  (c_k)_{k \in \Zdst} \in \liLam. $$
\end{enumerate}
\end{theorem}

\section{Corresponding Properties of Mild Distributions}

Abstract Harmonic Analysis views periodic functions as functions
on the torus group (assuming that the periodicity is known!) and provides
the classical setting of (orthogonal) Fourier series and discrete functions,
such as time-series as functions over $\Zdst$, which have a well-defined
DFT (Discrete Fourier transform). Functions which are the sum of
two incompatible periodic functions can still be treated by the
classical theory of Almost Periodic (AP) functions, using properly
defined {\it means}.

In contrast, the setting of mild distributions on $\Rdst$ allows to
view both special types of ``signals'' (independent of the periodicity)
as mild distributions, which have a Fourier transform in the $\SORd$-sense.
Let us just recall that this is easily defined by means of the classical
formula
$ \hatsi(f) = \sigma(\hatf), \quad f \in \SORd. $
This definition only makes use of the (isometric) Fourier invariance of $\SORd$
and also gives that it is isometric (and $\wwst$continuous) on $\SOPRdN$.
In this way it is clear that (due to the density of $\ScRd$ in $\SORdN$)
this is just the restriction of the generalized FT in the setting
of tempered distributions, which leaves $\SOPRd$ invariant. The
alternative, more elementary way using the sequential approach (as
described in \cite{fe20-1}) is of course equivalent.

% Let us describe one of the connections (without detailed proofs here) by 
% a typical case illustrating the general principle
%by concrete examples. % , usually close to the classical setting.
The connection to the classical setting is illustrated next.
\begin{lemma} \label{periodic-a1}
Assume that a locally integrable function $f \in \Liloc(\Rst)$ is $a$-periodic
for some $a > 0$. Then its Fourier transform (in the sense of $\SOPRd$) is
a weighted Dirac comb of the form
$$ \hatf = C_a \sumnZ  \hatf(n/a) \, \delta_{n/a}, \quad f \in \SORd, $$
where the sampling values $\hatf(n/a)$ can be identified (maybe up
to some constant) with the classical Fourier coefficients of the form
$$   F(n) :=  \int_{[0,a]} f(t) e^{2 \pi i n/a} dt, \quad n \in \Zst.$$
\end{lemma}

Next we turn to ``discrete signals''. Since both $\SORd$ and $\SOPRd$ (by
consequence) are invariant under affine transformations and arbitrary
discrete (cocompact) lattices $\Lambda \lhd \Rdst$ are of the form
$\Lambda = \Asp \ast \Zdst$, for some non-singular $d \times d$-matrix
$\Asp$ it suffices to formulate  the result  again (for convenience) just for $\Lambda = \Zdst$.

\section{Final Comments}

This paper is part of a series of notes with the joint goal of motivating
colleagues in mathematics and in the applied scientist to make use of the theory of mild distributions, which is best understood in the spirit of
THE Banach Gelfand Triple $\SOGTrRd$. It is part of the foundations of a
movement towards what the author would like to promote by the name of
``Conceptual Harmonic Analysis'',  a blend of abstract and applied Harmonic
Analysis, also preparing ground for the application of fast numerical methods.
Recent contributions emphasizing different aspects of this scenario are
\cite{fe24}, \cite{fe22-1}, \cite{fe24-1} or the more comprehensive
articles \cite{cofelu08} and \cite{feja22}.
Group theoretical aspects are also discussed in \cite{fegr92-2} and \cite{festch95}.

\nocite{gr01,rest00} 

\newcommand{\ORCID}{https://orcid.org/0000-0002-9927-0742}
\bibliographystyle{abbrv}  %%%%% NuHAG-BibTeX %%%%%%%
% \bibliographystyle{plain}

% \bibliography{D:/nuhagall/nuhagbib/nhgbib}

Address:

\FeiAdr

\ORCID

\end{document}